\documentclass[11pt,a4paper,reqno]{amsart}

\usepackage[T1]{fontenc}
\usepackage[utf8]{inputenc}
\usepackage[english]{babel}
\usepackage[margin=26mm]{geometry}
\usepackage{amsmath,amssymb,amsthm,mathtools}
\usepackage{microtype}
\usepackage{enumitem}
\usepackage[hidelinks,bookmarksopen=true,bookmarksnumbered=true]{hyperref}

\allowdisplaybreaks

\setlist[itemize]{leftmargin=*,topsep=0.4em,itemsep=0.2em}
\setlist[enumerate]{leftmargin=*,topsep=0.4em,itemsep=0.25em}

\newtheorem{theorem}{Theorem}[section]
\newtheorem{proposition}[theorem]{Proposition}
\newtheorem{lemma}[theorem]{Lemma}
\newtheorem{corollary}[theorem]{Corollary}
\theoremstyle{remark}
\newtheorem{remark}[theorem]{Remark}

\newcommand{\R}{\mathbb R}
\newcommand{\Ball}{\operatorname{B}}
\newcommand{\Area}{\operatorname{Area}}
\newcommand{\Vol}{\operatorname{Vol}}

\title{The rhombic dodecahedral conjecture}

\author{Annalisa Cesaroni}
\address{Dipartimento di Matematica ``Tullio Levi-Civita'', Universit\`a di Padova, Via Trieste 63, 35131 Padova, Italy}
\email{annalisa.cesaroni@unipd.it}
\author{Matteo Novaga}
\address{Dipartimento di Matematica, Universit\`a di Pisa, Largo Bruno Pontecorvo 5, 56127 Pisa, Italy}
\email{matteo.novaga@unipi.it}


\date{\today}

\begin{document}

\begin{abstract}
We prove that every three-dimensional parallelohedron $P$ satisfies
\[
  \mathcal H^2(\partial P)\ge 12\sqrt2\,r(P)^2,
\]
where $r(P)$ is the inradius of $P$, with equality if and only if $P$ is similar to the regular rhombic dodecahedron.
\end{abstract}

\maketitle
\tableofcontents

\section{Introduction}\label{sec:introduction}

A three-dimensional parallelohedron is a convex polytope whose translates tile $\R^3$.  By the classical results of Venkov~\cite{Venkov1954} and McMullen~\cite{McMullen1980}, every such polytope admits a face-to-face lattice tiling.  In dimension three there are five Fedorov types: parallelepipeds, hexagonal prisms, rhombic dodecahedra, elongated rhombic dodecahedra, and truncated octahedra; see, for instance, L\'angi~\cite[Section~2.1]{Langi2022}.

In this paper we consider the minimal surface-area problem at fixed inradius.  
L\'angi states the expected sharp inequality as Conjecture~2 in~\cite{Langi2022}, under the name \emph{Rhombic Dodecahedral Conjecture}, and attributes it to Bezdek.  Related sharp results for rhombic dodecahedra and for Voronoi cells of lattice sphere packings were obtained in~\cite{Bezdek2000,BezdekDaroczyLiu1999}.

For a convex body $P\subset\R^3$ we shall write
\[
  \Vol(P):=|P|,\qquad
  \Area(P):=\mathcal H^2(\partial P),\qquad
  r(P):=\sup\{r>0:\exists x\in\R^3,\ \Ball(x,r)\subset P\}.
\]

Our main result is the proof of the Rhombic Dodecahedral Conjecture:

\begin{theorem}[Main result]\label{thm:main}
Let $P$ be a three-dimensional parallelohedron.  Then
\begin{equation}\label{eq:mainhom}
  \Area(P)\ge 12\sqrt2\,\,r(P)^2.
\end{equation}
Equality holds if and only if $P$ is similar to the regular rhombic dodecahedron.  In particular, every parallelohedron containing a unit ball has surface area at least $12\sqrt2$.
\end{theorem}

In~\cite{CesaroniNovaga2025} the authors proved existence and regularity results for minimizers of the minimal surface-area problem at fixed inradius among periodic fundamental domains, and in particular showed that a minimizer cannot be a parallelohedron.  The present result concerns the same minimization problem in the constrained class of convex translative tiles, which coincides with the class of parallelohedra.

Other related results have appeared recently. 
L\'angi~\cite[Theorem~1]{Langi2022} proved that the regular truncated octahedron minimizes mean width among unit-volume parallelohedra. 
Recently, Hales and Song~\cite[Theorem~1.1]{HalesSong2026}
and, independently, the authors~\cite[Theorem~1.1]{CesaroniNovaga2026} proved that the regular truncated octahedron uniquely minimizes surface area among parallelohedra with fixed volume.  The latter work also proves the stronger sharp bound on the non-truncated Fedorov strata, with the rhombic dodecahedron as unique minimizer.  More precisely, every three-dimensional parallelohedron satisfies  
\begin{equation}\label{sharpbcc} \Area(P) \ge \frac{3(1+2\sqrt{3})}{4^{2/3}}\Vol(P)^{2/3},
\end{equation} with equality if and only if $P$ is similar to the truncated octahedron;  
moreover, if $P$ is  not of truncated-octahedral Fedorov type, then
\begin{equation}\label{sharpfcc} \Area(P) \ge 3\,2^{5/6}\Vol(P)^{2/3},
\end{equation}  
with equality if and only if $P$ is similar to the rhombic dodecahedron. 

The fixed-volume theorem does not by itself imply Theorem~\ref{thm:main}, at least for truncated-octahedral Fedorov type parallelohedra.  Indeed, by the Venkov--McMullen theorem the translates of $P$ may be taken to form a lattice tiling.  Since a parallelohedron is centrally symmetric, convexity allows an inball of radius $r(P)$ to be centered at its symmetry center; its lattice translates therefore form a lattice sphere packing.  Gauss's theorem on three-dimensional lattice sphere packings~\cite{Gauss1831} gives the optimal lattice-packing density $\pi/\sqrt{18}=\pi/(3\sqrt2)$.  Hence
\begin{equation}\label{eq:GaussVolume}
  \Vol(P)\ge 4\sqrt2\,r(P)^3.
\end{equation}
Combining estimate \eqref{eq:GaussVolume} with the sharp BCC isoperimetric constant \eqref{sharpbcc} yields only
\[
  \Area(P)\ge 3\,2^{1/3}(1+2\sqrt3)\,r(P)^2
  \approx 16.873247\,r(P)^2,
\]
whereas
\[
  12\sqrt2\,r(P)^2\approx16.970563\,r(P)^2,
\]
so the fixed-volume argument misses the conjectured constant by about $0.57\%$.  This is not the case for the non-truncated Fedorov types, for which we may use the sharp isoperimetric bound \eqref{sharpfcc} in place of \eqref{sharpbcc}, and conclude by Gauss's lattice-packing bound (see Section~\ref{sec:boundary}).

Therefore, the only remaining issue is the proof of the conjecture for the truncated-octahedral Fedorov stratum, where we need an appropriate estimate taking into account the distances of the facets from the center.

Finally, we point out that the Rhombic Dodecahedral Conjecture we prove in Theorem~\ref{thm:main} should be distinguished from the classical Dodecahedral Conjecture and its strong form.  Hales--McLaughlin in~\cite{HalesMcLaughlin2010} proved that every Voronoi cell in a packing of unit balls has volume at least that of the regular dodecahedron of inradius $1$.  Bezdek's Strong Dodecahedral Conjecture replaces volume by surface area; Hales proved that every bounded Voronoi cell in such a packing has surface area at least that of the regular dodecahedron of inradius $1$, approximately
$16.6508$, see~\cite{BezdekRogersLemma2000,HalesStrong2013}.  This is strictly below
$12\sqrt2\approx16.970563$.
Generic Voronoi cells in a sphere packing do form a tiling, but they need not be translates of one common cell.  The larger constant in Theorem~\ref{thm:main} therefore is due to the additional parallelohedral structure of the problem, rather than tiling alone.

\smallskip
\noindent{\bf Plan of the paper.}
Section~\ref{sec:boundary} contains the result for the four non-truncated Fedorov types.  
Section~\ref{sec:K4} introduces the graphical representation $K_4$ and derives the surface-area and inradius formulas.  In particular, it gives an equivalent formulation of the ratio $\Area(P)/r(P)^2$ for truncated-octahedral parallelohedra.   Section~\ref{sec:descent} contains the proof of the key negative-descent property at fixed affine metric, see Proposition~\ref{prop:descent}.  Section~\ref{sec:contacts} analyzes the six-contact configuration, introduces the four-point metric, and provides a quantitative refinement of the four-point determinant remainder \eqref{eq:quantremaindernew}. Section~\ref{sec:deficit} derives from it  the determinant--tangency deficit estimate. Section~\ref{sec:mainproof} completes the proof of Theorem~\ref{thm:main}.

\smallskip
\noindent{\bf AI statement.} ChatGPT (GPT-5.6 Sol) was used during the preparation of this work to check algebraic identities and to assist in exploring proof arguments.  The authors checked the final statements and proofs and take responsibility for the content of the manuscript.

\section{The non-truncated strata}
\label{sec:boundary}
We first prove Theorem~\ref{thm:main} on the non-truncated Fedorov strata. This result is a direct consequence of the sharp isoperimetric estimate \eqref{sharpfcc} and of Gauss's lattice-packing bound \eqref{eq:GaussVolume}.

\begin{proposition}\label{prop:boundaryRDC}
If $P$ is a three-dimensional parallelohedron that is not of truncated-octahedral Fedorov type, then
\[
  \Area(P)\ge12\sqrt2\,r(P)^2,
\]
with equality if and only if $P$ is similar to the regular rhombic dodecahedron.
\end{proposition}

\begin{proof}
By \cite[Theorem~1.1]{CesaroniNovaga2026}, every non-truncated Fedorov type satisfies the sharp estimate \eqref{sharpfcc}, with equality only for the regular rhombic dodecahedron, up to similarities.  Since every parallelohedron admits a lattice tiling~\cite{Venkov1954,McMullen1980} and is centrally symmetric, an inball of radius $r=r(P)$ may be centered at its symmetry center; its lattice translates give a lattice packing by balls of radius $r(P)$.  Gauss's lattice-packing theorem~\cite{Gauss1831} gives the optimal three-dimensional lattice-packing density $\pi/\sqrt{18}=\pi/(3\sqrt2)$, attained by the FCC lattice. Hence
\[
  \Vol(P)\ge
  \frac{(4\pi/3)r^3(P)}{\pi/(3\sqrt2)}=4\sqrt2\,r^3(P).
\]
Therefore, using \eqref{sharpfcc}, we get 
\begin{align*}
  \Area(P)
  &\ge 3\,2^{5/6}\Vol(P)^{2/3}\ge 3\,2^{5/6}(4\sqrt2)^{2/3}r^2(P)
   =12\sqrt2\,r^2(P).
\end{align*}
If equality holds, then equality also holds in~\eqref{sharpfcc}; hence $P$ is similar to the regular rhombic dodecahedron.  Conversely, the regular rhombic dodecahedron is the Voronoi cell of the FCC lattice, so its centered inballs realize equality in Gauss's lattice-packing bound, while equality in~\eqref{sharpfcc} holds by \cite[Theorem~1.1]{CesaroniNovaga2026}.  Thus it realizes equality in the displayed estimate.
\end{proof}

\section{Graphical representation of parallelohedra}
\label{sec:K4}
Following \cite[Sections~3 and~6]{CesaroniNovaga2026}, we use a graphical representation of parallelohedra to derive explicit formulas for the area and, in the truncated-octahedral case, the inradius.  We then reformulate the ratio $\Area(P)/r(P)^2$ in these coordinates.

We first recall the Selling--Voronoi coordinates used in \cite[Section~3]{CesaroniNovaga2026}.  Every full-rank lattice in $\R^3$ can be associated with a nonnegative edge-weighted realization of $K_4$, the complete graph with four vertices $\{0,1,2,3\}$.  
We label the vertices by the four vectors $v_0,v_1,v_2,v_3$ of the obtuse superbase and identify the Selling parameter $-v_i\cdot v_j$ with the weight of the edge $ij$ of $K_4$.    

We use the edge order
$(01,02,03,12,13,23)$, and write the six nonnegative edge weights as
$\omega=(a,b,c,d,e,f)$,
that is $a=-v_0\cdot v_1\geq 0$, $b=-v_0\cdot v_2\geq 0$ and so on.

To realize the abstract graph $K_4$ in a fixed three-dimensional root space, introduce auxiliary points
$p_0=0$, $p_1=\mathbf e_1$, $p_2=\mathbf e_2$, $p_3=\mathbf e_3$.  These $p_i$ are not the vectors of the obtuse superbase.  We take the incidence roots
\begin{equation}\label{eq:roots}
\begin{aligned}
 r_a&=\mathbf e_1, & r_b&=\mathbf e_2, & r_c&=\mathbf e_3,\\
 r_d&=\mathbf e_1-\mathbf e_2, & r_e&=\mathbf e_1-\mathbf e_3, & r_f&=\mathbf e_2-\mathbf e_3.
\end{aligned}
\end{equation}
The centered graphical zonotope is
\begin{equation}\label{eq:Zomega}
  Z_\omega:=\sum_{\xi\in\{a,b,c,d,e,f\}} \left[-\frac{\xi r_\xi}{2},\frac{\xi r_\xi}{2}\right].
\end{equation}  
Define the associated symmetric matrix
\begin{equation}\label{eq:Aomega}
  A(\omega):=\sum_{\xi\in\{a,b,c,d,e,f\}} \xi\, r_\xi r_\xi^T.
\end{equation}
For Selling parameters this is the Gram matrix appearing in \cite[Proposition~3.1]{CesaroniNovaga2026}.

We shall use the following representation statement.

\begin{proposition}\label{prop:CNrepresentation}
Every three-dimensional parallelohedron is, up to translation, of the form
\[
  P=LZ_\omega,
\]
where $L\in GL(3,\R)$ and $Z_\omega$ is the centered graphical zonotope of a connected nonnegative weighting $\omega$ of the six edges of $K_4$.  Conversely, for every $L\in GL(3,\R)$ and every connected nonnegative weighting $\omega$, the body $LZ_\omega$ is a three-dimensional parallelohedron.
\end{proposition}

\begin{proof}
The first assertion is \cite[Theorem~6.1]{CesaroniNovaga2026}.  For the converse, let $G_\omega$ be the support graph formed by the edges with positive weight.  Since $G_\omega$ is connected, it contains a spanning tree.  The three incidence roots corresponding to the edges of any spanning tree of $K_4$ are linearly independent.  Hence, for every nonzero $x\in\R^3$,
\[
 x^TA(\omega)x=\sum_\xi \xi\,(x\cdot r_\xi)^2>0,
\]
so $A(\omega)$ is positive definite.  By the converse part of the Selling--Voronoi reduction in \cite[Proposition~3.1]{CesaroniNovaga2026}, the six weights determine a full-rank lattice $\Lambda$ with Gram matrix $A(\omega)$, up to orthogonal isometry.  Then \cite[Proposition~3.2]{CesaroniNovaga2026} gives
\[
  \operatorname{Vor}(\Lambda)=A(\omega)^{-1/2}Z_\omega
\]
up to an orthogonal change of coordinates.  Thus $Z_\omega$ is an invertible linear image of a lattice Voronoi cell and therefore tiles by translations of a lattice.  Applying the invertible map $L$ preserves this property, so $LZ_\omega$ is a three-dimensional parallelohedron.
\end{proof}

Modulo complementation there are seven nontrivial cuts of $K_4$,  four singleton cuts and three $2+2$ cuts that we order as follows:  $\{0\}\mid \{1,2,3\}$, $\{1\}\mid \{0,2,3\}$, $\{2\}\mid \{1,0,3\}$, $\{3\}\mid \{1,2,0\}$,
$\{0,1\}\mid\{2,3\}$, $\{0,2\}\mid\{1,3\}$, and
$\{0,3\}\mid\{1,2\}$.  Their normal vectors are
\begin{equation}\label{eq:z}
\begin{aligned}
 z_0&=(1,1,1)^T, &z_1&=(1,0,0)^T, &z_2&=(0,1,0)^T, &z_3&=(0,0,1)^T,\\
 z_4&=(0,1,1)^T, &z_5&=(1,0,1)^T, &z_6&=(1,1,0)^T.
\end{aligned}
\end{equation}
 For  $S\subset\{0,1,2,3\}$, the weight of the cut $S\mid S^c$ is given by the sum of the weights of the edges   with one endpoint in $S$ and one in $S^c$. 
 In our case the   cut weights are
\begin{equation}\label{eq:ci}
\begin{aligned}
 c_0(\omega)&=a+b+c, &c_1(\omega)&=a+d+e, &c_2(\omega)&=b+d+f, &
 c_3(\omega)&=c+e+f, \\c_4(\omega)&=b+c+d+e, 
 &c_5(\omega)&=a+c+d+f, &c_6(\omega)&=a+b+e+f.&&
\end{aligned}
\end{equation}
 
The seven cut weights satisfy the linear relation
\begin{equation}\label{eq:cutrelation}
  c_0(\omega)+c_1(\omega)+c_2(\omega)+c_3(\omega)=c_4(\omega)+c_5(\omega)+c_6(\omega).
\end{equation}
The facet-area polynomials are
\begin{equation}\label{eq:qi}
\begin{aligned}
 q_0(\omega)&=de+df+ef, &q_1(\omega)&=bc+bf+cf, &q_2(\omega)&=ac+ae+ce, 
& q_3(\omega)&=ab+ad+bd,\\ q_4(\omega)&=af, &q_5(\omega)&=be, &q_6(\omega)&=cd.&&
\end{aligned}
\end{equation}
Finally, the weighted spanning-tree polynomial of $K_4$ coincides with the volume of $Z_\omega$ and is given by 
\begin{align}
\Vol(Z_\omega)=D(\omega)={}&abc+abe+abf+acd+acf+ade+adf+aef\notag\\
&+bcd+bce+bde+bdf+bef+cde+cdf+cef.
\label{eq:D}
\end{align}

The matrix-tree theorem gives the key identity
\begin{equation}\label{eq:EulerIdentity}
  \sum_{i=0}^6q_i(\omega) c_i(\omega)=3D(\omega)=3\Vol(Z_\omega).
\end{equation}

Let $P=LZ_\omega$ be as in Proposition~\ref{prop:CNrepresentation}, where $L\in GL(3,\R)$, and define 
\begin{equation}\label{eq:lambdaQ}
  \lambda:=|\det L|^{1/3},\qquad
  Q:=|\det L|^{2/3}L^{-1}L^{-T}.
\end{equation}
Note that $Q$ is positive definite and $\det Q=1$.  Moreover, the polar decomposition gives
\[
  L=\lambda OQ^{-1/2}
\]
for some orthogonal matrix $O$.  Thus $LZ_\omega$ is similar to $Q^{-1/2}Z_\omega$.  Recalling \eqref{eq:z}, we define
\begin{equation}\label{eq:wiN}
  w_i:=\sqrt{z_i^TQz_i},\qquad
  N(\omega,Q):=\sum_{i=0}^6q_i(\omega)w_i= \sum_{i=0}^6q_i(\omega)\sqrt{z_i^TQz_i}.
\end{equation}
In what follows we write $\|x\|_Q:=\sqrt{x^TQx}$.
\begin{remark}\label{rem:S4}
We note that the group $S_4$
acts on Selling coordinates and on the seven cuts by relabelling the four vertices of the
superbase.   If the incidence roots are written, up to orientation, as differences of the auxiliary points $p_i$, a permutation $\sigma\in S_4$ sends each root to the corresponding relabelled root; with the fixed orientations in~\eqref{eq:roots}, this is one of the six listed roots up to sign.  Thus the induced linear map on the three-dimensional root space is integral and unimodular.  Dually it permutes the seven cut normals up to the corresponding change of basis.  When the metric is represented by a matrix $Q$, the same relabelling is accompanied by the corresponding congruence change of $Q$; positive definiteness and the condition $\det Q=1$ are preserved.  Consequently, statements quantified over all such $Q$ may be reduced by $S_4$ symmetry, with $c_i,q_i,w_i$ and the corresponding constraints relabelled simultaneously.
\end{remark}

We can now compute the volume and area of $P$ explicitly in terms of $L$ and $\omega$.
\begin{lemma}\label{lem:surface}
For $P=LZ_\omega$,
\begin{equation}\label{eq:surface}
  \Area(P)=2\lambda^2N(\omega,Q),\qquad \Vol(P)=\lambda^3\Vol(Z_\omega).
\end{equation}
\end{lemma}

\begin{proof}
The proof is contained in~\cite[Proposition~6.2]{CesaroniNovaga2026}; we recall the main calculation. 
Each  facet of $Z_\omega$ with normal $z_i$ has area $q_i(\omega)|z_i|$.  Under $L$, an area element in a plane with normal $z_i$ is multiplied by
\[
  |\det L|\frac{|L^{-T}z_i|}{|z_i|}.
\]
Using  \eqref{eq:lambdaQ},
\[
  |L^{-T}z_i|=\lambda^{-1}\sqrt{z_i^TQz_i}=\lambda^{-1}w_i.
\]
 Hence one facet in the $i$-th opposite pair has area $\lambda^2q_iw_i$.  Summing over the seven opposite pairs gives the first formula, while the second is trivial.  
\end{proof}

We now compute the inradius of truncated-octahedral parallelohedra. 
\begin{lemma}\label{lem:inradius}
Assume that all six graphical weights are positive (i.e.  $a,b,c,d,e,f>0$), so all seven cut directions support facet pairs.  Then the support function $h_{Z_{\omega}}$ of the zonotope $Z_\omega$ in the direction $z_i$ coincides with 
\begin{equation}\label{eq:support}
 h_{Z_\omega}(z_i)=\frac{c_i(\omega)}{2},
\end{equation}
and the inradius of $P=LZ_{\omega}$ is given by 
\begin{equation}\label{eq:inradius}
 r(P)=\frac{\lambda}{2}\min_{0\le i\le6}\frac{c_i(\omega)}{w_i}
     =\frac{\lambda}{2\mu},
 \qquad
 \mu:=\max_{0\le i\le6}\frac{w_i}{c_i(\omega)}.
\end{equation}
\end{lemma}

\begin{proof}
By the definition of the support function, 
\[ 
 h_{Z_\omega}(z_i)=\sum_{\xi\in \{a,b,c,d,e,f\}}  \frac{\xi}{2}|z_i\cdot r_\xi|,
\]
where $\omega=(a,b,c,d,e,f)$.
The incidence-root/cut construction gives $|z_i\cdot r_\xi|=1$ exactly when edge $\xi$ crosses cut $i$, and $0$ otherwise.  This implies  \eqref{eq:support}.

After application of  $L$, the corresponding supporting hyperplane has normal $L^{-T}z_i$ and support value $c_i(\omega)/2$.  The distance from the origin  is therefore
\[
  \frac{c_i(\omega)}{2|L^{-T}z_i|}=\frac{\lambda c_i(\omega)}{2w_i}.
\]
It remains to observe that the largest inscribed ball can be centered at the center of symmetry of $P$.  Indeed, if $P=-P$ contains $\Ball(x,r)$, then it also contains $\Ball(-x,r)$.  If $|y|\le r$, both $x+y$ and $-x+y$ belong to $P$, and convexity gives $y\in P$.  Thus $\Ball(0,r)\subset P$, and the inradius is the minimum of the central facet distances.
\end{proof}

Combining Lemmata~\ref{lem:surface} and~\ref{lem:inradius}, we obtain  
\begin{equation}\label{eq:ratioMuN}
  \frac{\Area(P)}{r(P)^2}=8\mu^2N(\omega,Q).
\end{equation}

Note that the quantity $\mu^2N(\omega, Q)$ is invariant under the rescaling $\omega\mapsto t\omega$, $t>0$, because
$c_i(\omega)\mapsto t c_i(\omega)$, $q_i(\omega)\mapsto t^2q_i(\omega)$, $\mu\mapsto t^{-1}\mu$, and $N(\omega,Q)\mapsto t^2N(\omega,Q)$. 
Therefore, in the six-generator case (i.e.  $a,b,c,d,e,f>0$) we may normalize to 
\begin{equation}\label{eq:constraints}
  \mu=1,\qquad \text{equivalently}\qquad
  c_i(\omega)\ge w_i=\sqrt{z_i^TQz_i}\quad\text{for }0\le i\le6,
\end{equation}
with equality for at least one index.  For the representative $P=Q^{-1/2}Z_\omega$, this normalization gives $r(P)=1/2$ and $\Area(P)=2N(\omega,Q)$, so Theorem~\ref{thm:main} reduces for truncated-octahedral parallelohedra to prove that 
\begin{equation}\label{eq:Ntarget}
  N(\omega, Q)\ge\frac{3}{\sqrt2}.
\end{equation}
\begin{remark}
\label{rem:boundarysupport}
The support identity \eqref{eq:support} remains valid when some graphical weights vanish.  Assume that the support graph is connected, so that $Z_\omega$ is three-dimensional, and let
\[
 F(\omega):=\{i\in\{0,\ldots,6\}:q_i(\omega)>0\}
\]
be the set of actual facet directions.  Indeed, the cut description of the facets in \cite[Proposition~3.3]{CesaroniNovaga2026}, together with the area computation in the proof of Lemma~\ref{lem:surface}, shows that the facet normals of $Z_\omega$ are among the seven cut directions $z_i$, and that the facet pair with normal $z_i$ has area $q_i(\omega)|z_i|$.  Hence it is an actual facet pair precisely when $q_i(\omega)>0$.  Moreover, the centering argument in the proof of Lemma~\ref{lem:inradius} uses only the central symmetry and the convexity of $LZ_\omega$, hence applies verbatim here.  The inradius is then the minimum of the central distances over the directions in $F(\omega)$ only:
\[
 r(LZ_\omega)=\frac{\lambda}{2}\min_{i\in F(\omega)}\frac{c_i(\omega)}{w_i}.
\]
Consequently, if $c_i(\omega)\ge w_i$ for every $i$, then $r(LZ_\omega)\ge\lambda/2$.  Since Lemma~\ref{lem:surface} still gives $\Area(LZ_\omega)=2\lambda^2N(\omega,Q)$, we obtain
\[
 \frac{\Area(LZ_\omega)}{r(LZ_\omega)^2}\le 8N(\omega,Q).
\]
\end{remark}

\section{The negative descent property}
\label{sec:descent}

We now work at fixed affine metric, equivalently at fixed $Q$. 
For a fixed matrix $Q$, the function $N(\omega, Q)$ in \eqref{eq:wiN} is a homogeneous quadratic form in the six graphical weights $a,b,c,d,e,f$, while the constraints \eqref{eq:constraints} are linear. Indeed, writing for simplicity $N(\omega)=N(\omega,Q)$ since $Q$ is fixed, and  using \eqref{eq:qi}, we get
\begin{align}
N(\omega)={}&w_0(de+df+ef)+w_1(bc+bf+cf)\notag\\
&+w_2(ac+ae+ce)+w_3(ab+ad+bd)\notag\\
&+w_4af+w_5be+w_6cd.
\label{eq:Nquadratic}
\end{align}
For $\eta\in\R^6$ we also have 
\begin{equation}\label{eq:quadraticline}
  N(\omega+t\eta)+N(\omega-t\eta)-2N(\omega)=2t^2N(\eta).
\end{equation}

We have the following negative descent property of the homogeneous quadratic form $N$ under the constraint \eqref{eq:constraints}.
\begin{proposition}\label{prop:descent}
Fix $Q>0$ with $\det Q=1$.  Let $\omega=(a,b,c,d,e,f)$ with $a,b,c,d,e,f>0$ satisfy $c_i(\omega)\ge w_i$ for all $i$.  Then there exists a feasible weighting $\omega^*$ with $c_i(\omega^*)\ge w_i$ and 
\[
 N(\omega^*,Q)\le N(\omega,Q)
\]
such that 
either at least one graphical weight of $\omega^*$ is zero or  exactly six of the seven constraints $c_i(\omega^*)\ge w_i$ are equalities.  Moreover, if at most five of the seven constraints are active at $\omega$, then the inequality is strict,
\[
 N(\omega^*,Q)<N(\omega,Q).
\]
\end{proposition}
 The proof of Proposition~\ref{prop:descent} is based on a finite-case lemma.  We first introduce the cut matrix $\mathsf C$.   
 
Let $\mathsf C$ be the $7\times6$ matrix satisfying $(c_0(\omega),\dots,c_6(\omega))^T=\mathsf C (a,b,c,d,e,f)^T$.  Explicitly,
\begin{equation}\label{eq:cutmatrix}
\mathsf C =
\begin{pmatrix}
1&1&1&0&0&0\\
1&0&0&1&1&0\\
0&1&0&1&0&1\\
0&0&1&0&1&1\\
0&1&1&1&1&0\\
1&0&1&1&0&1\\
1&1&0&0&1&1
\end{pmatrix}.
\end{equation}
   If $\mathsf C^{(i)}$ denotes the $6\times6$ matrix obtained by deleting row $i$, direct elimination gives
\[
 \bigl(\det \mathsf C^{(0)},\det \mathsf C^{(1)},\ldots,\det \mathsf C^{(6)}\bigr)
 =(8,-8,8,-8,-8,8,-8).
\]
In particular no six rows are dependent, i.e. 
$ \operatorname{rank}\mathsf C=6$.  
Its left kernel is generated by
\[
  (-1,-1,-1,-1,1,1,1),
\]
which coincides with equation \eqref{eq:cutrelation}.
 
Let now  $w_i=\|z_i\|_Q$, as introduced in \eqref{eq:wiN}. 
We recall the Hornich--Hlawka inequality, see \cite{Hornich1942}, which holds for every three vectors $x,y,z$ in a real inner-product space:
\begin{equation}\label{eq:Hlawka}
 \|x+y+z\|+\|x\|+\|y\|+\|z\|
 \ge \|x+y\|+\|y+z\|+\|z+x\|.
\end{equation}
For completeness, we record the factorization that will also give the strict inequality needed below.  Set
\[
 \alpha=\|x\|,\quad \beta=\|y\|,\quad \gamma=\|z\|,\quad
 \delta=\|x+y+z\|,
\]
\[
 X=\|x+y\|,\quad Y=\|y+z\|,\quad Z=\|z+x\|.
\]
Expanding the inner products gives
\begin{equation}\label{eq:Hlawkasquares}
 \alpha^2+\beta^2+\gamma^2+\delta^2=X^2+Y^2+Z^2.
\end{equation}
Writing $S=\alpha+\beta+\gamma+\delta$ and $T=X+Y+Z$, a direct expansion using \eqref{eq:Hlawkasquares} yields
\begin{align}\label{eq:Hlawkafactor}
 S(S-T)={}&(\alpha+\beta-X)(\gamma+\delta-X)\notag\\
 &+(\beta+\gamma-Y)(\alpha+\delta-Y)\notag\\
 &+(\gamma+\alpha-Z)(\beta+\delta-Z).
\end{align}
Every factor on the right-hand side is nonnegative by the triangle inequality.  If $S=0$, all three vectors vanish and \eqref{eq:Hlawka} is trivial; otherwise $S>0$, so \eqref{eq:Hlawkafactor} implies $S\ge T$ and proves \eqref{eq:Hlawka}.

We now apply \eqref{eq:Hlawka} with $x=z_1$, $y=z_2$, $z=z_3$ in the norm $\|\cdot\|_Q$.  Since
\[
 z_0=z_1+z_2+z_3,\qquad z_4=z_2+z_3,\qquad
 z_5=z_1+z_3,\qquad z_6=z_1+z_2,
\]
we obtain the non-strict Hornich--Hlawka inequality for the $w_i$.  It is strict in this particular configuration.  Indeed, in the first product on the right-hand side of \eqref{eq:Hlawkafactor},
\[
 \|z_1\|_Q+\|z_2\|_Q-\|z_1+z_2\|_Q>0
\]
because $z_1$ and $z_2$ are linearly independent, while
\[
 \|z_3\|_Q+\|z_0\|_Q-\|z_0-z_3\|_Q>0
\]
because equality would force $z_3$ and $-z_0$ to be positively collinear, which they are not.  Hence that product is strictly positive and therefore
\begin{equation}\label{eq:Hlawkaw}
  w_0+w_1+w_2+w_3> w_4+w_5+w_6.
\end{equation}

We now prove the negative descent lemma. 
\begin{lemma}\label{lem:negative}
Let $J\subset\{0,\dots,6\}$ contain five cut indices and let $\mathsf C_J$ be the $5\times6$ matrix obtained from $\mathsf C$ by deleting the rows whose indices are not in $J$. Then there exists $\eta\in\R^6$  which satisfies 
\begin{equation}\label{eq:NetaNegative}
  \mathsf C_J\eta=0\qquad \text{and}\qquad N(\eta)<0.
\end{equation}
\end{lemma}

\begin{proof}
By the $S_4$ symmetry described in Remark~\ref{rem:S4} it is enough to treat one representative of each orbit of the natural $S_4$ action on the unordered pairs of omitted cuts.  There are exactly three such orbits, according to the type of the omitted pair.  Indeed, $S_4$ is $2$-transitive on the four singleton cuts, so all pairs of singleton cuts form one orbit; the natural surjection $S_4\to S_3$ onto the permutations of the three $2+2$ cuts is $2$-transitive on them, so all pairs of $2+2$ cuts form one orbit; finally, for a mixed pair, the stabilizer of a vertex is isomorphic to $S_3$ acting transitively on the three $2+2$ cuts, so the orbit of a mixed pair has $4\cdot3=12$ elements, that is, all of them.  

\medskip
\noindent\textbf{Case 1: the omitted cuts are $5$ and $6$ (two $2+2$ cuts).}
Choose
\[
  \eta=(0,1,-1,-1,1,0)\qquad\text{in the coordinates }(a,b,c,d,e,f).
\]
Using \eqref{eq:cutmatrix} we get
$ 
  \mathsf C\eta=(0,0,0,0,0,-2,2)^T,$ which implies that 
  $\mathsf C_J\eta=0$.  Substitution into \eqref{eq:Nquadratic} and using   \eqref{eq:Hlawkaw}, gives
\[
 N(\eta)=-w_0-w_1-w_2-w_3+w_5+w_6<-w_4<0.
\]

\medskip
\noindent\textbf{Case 2: the omitted cuts are $0$ and $4$ (one singleton and one $2+2$ cut).}
Choose 
\[
  \eta=(0,1,1,0,0,-1).
\]
Then
$ \mathsf C\eta=(2,0,0,0,2,0,0)^T$, so $\mathsf C_J\eta=0$
and a direct substitution gives
$N(\eta)=-w_1<0$.

\medskip
\noindent\textbf{Case 3: the omitted cuts are $2$ and $3$ (two singleton cuts).}
Choose 
\[
  \eta=(0,-1,1,-1,1,0).
\]
Then
$  \mathsf C\eta=(0,0,-2,2,0,0,0)^T
$ so $\mathsf C_J\eta=0$ and 
\begin{equation}\label{eq:NetaCase3}
  N(\eta)=-w_0-w_1+w_2+w_3-w_5-w_6.
\end{equation}
If $w_5\le w_6$, then the strict triangle inequalities
\[
 w_2=\|z_0-z_5\|_Q<w_0+w_5,
 \qquad
 w_3=\|z_5-z_1\|_Q<w_5+w_1
\]
hold because the corresponding vector pairs are nonparallel.  Hence
\[
 w_2+w_3<w_0+w_1+2w_5
 \le w_0+w_1+w_5+w_6,
\]
which makes \eqref{eq:NetaCase3} negative.  If $w_6\le w_5$, use instead
\[
 w_2=\|z_6-z_1\|_Q<w_6+w_1,
 \qquad
 w_3=\|z_0-z_6\|_Q<w_0+w_6,
\]
to obtain the same conclusion.

These three cases exhaust the $S_4$ orbits and prove the lemma.
\end{proof}

\begin{proof}[Proof of Proposition~\ref{prop:descent}]
All seven constraints \eqref{eq:constraints} cannot be equalities, because the cut identity~\eqref{eq:cutrelation} together with $c_i(\omega)=w_i$ for every $i$ would contradict the strict inequality~\eqref{eq:Hlawkaw}.  Let $\mathcal A$ be the set of active cut constraints.  If $|\mathcal A|=6$, choose $\omega^*=\omega$ and we are done.  Assume therefore that $|\mathcal A|\le5$.

Choose a set $J$ of five cut indices with $\mathcal A\subset J$.  Since every five rows of $\mathsf C$ are independent, $\ker\mathsf C_J$ is one-dimensional.  Lemma~\ref{lem:negative} provides a nonzero vector $\eta_0\in\ker\mathsf C_J$ with $N(\eta_0)<0$.  Hence every nonzero generator $\eta=t\eta_0$ of this kernel satisfies
\[
 N(\eta)=t^2N(\eta_0)<0.
\]
Fix such a generator $\eta$ and consider the maximal feasible interval
\[
 I_\omega:=\{t\in\R:(\omega+t\eta)_j\ge0\ \text{for every graphical coordinate }j,
 \quad c_i(\omega+t\eta)\ge w_i\ \text{for all }i\}.
\]
Because all graphical weights of $\omega$ are positive and every constraint outside $\mathcal A$ has positive slack, $0$ belongs to the interior of $I_\omega$.

We next justify that $I_\omega$ is compact.  A nonzero element of $\ker\mathsf C_J$ cannot have all coordinates nonnegative: if $\eta\ge0$ and $\mathsf C_J\eta=0$, then each of the five cut sums is zero, and each column of \eqref{eq:cutmatrix} contains four entries equal to $1$, so the supports of any five rows cover all six graphical coordinates; hence every coordinate of $\eta$ is zero.  Applying the same argument to $-\eta$ shows that a nonzero kernel vector cannot have all coordinates nonpositive either.  Thus $\eta$ has at least one positive and one negative coordinate.  The constraints $\omega+t\eta\ge0$ therefore bound $t$ from both sides.  Since the feasible conditions are closed, $I_\omega=[t_-,t_+]$ is a compact nondegenerate interval with $t_-<0<t_+$.

By \eqref{eq:quadraticline}, the function $t\mapsto N(\omega+t\eta)$ is a quadratic with second derivative $2N(\eta)<0$, hence it is strictly concave on $I_\omega$.  Its minimum is therefore attained at an endpoint.  Choose $t_*\in\{t_-,t_+\}$ realizing this minimum and put $\omega^*=\omega+t_*\eta$.  Since $t_-<0<t_+$, writing
\[
 0=\theta t_-+(1-\theta)t_+,\qquad
 \theta=\frac{t_+}{t_+-t_-}\in(0,1),
\]
strict concavity gives
\begin{equation}\label{eq:strictstep}
 N(\omega)>\theta N(\omega+t_-\eta)+(1-\theta)N(\omega+t_+\eta)
 \ge\min\{N(\omega+t_-\eta),N(\omega+t_+\eta)\}=N(\omega^*),
\end{equation}
so each such step strictly decreases $N$.

At an endpoint of the maximal feasible interval, either a graphical weight vanishes or a previously inactive cut inequality becomes an equality.  If a graphical weight vanishes, we are done.  Otherwise all graphical weights of $\omega^*$ remain positive.  For every $i\in J$ we have $c_i(\omega^*)=c_i(\omega)$ because $\mathsf C_J\eta=0$, so any newly active cut constraint has index outside $J$.  Thus the number of active cut constraints strictly increases.  If it is still below six, repeat the construction.  All seven constraints can never be active by \eqref{eq:Hlawkaw}, so after finitely many steps the process reaches either a zero graphical weight or exactly six active cut constraints.  By \eqref{eq:strictstep}, $N$ strictly decreases at each step, so the final weighting $\omega^*$ satisfies $N(\omega^*,Q)<N(\omega,Q)$ whenever at least one step has been performed, that is, whenever at most five constraints are active at $\omega$.
\end{proof}

 The following corollary handles the boundary endpoint of the descent.

\begin{corollary}\label{cor:boundaryfull}
Let $\omega^*$ be as in Proposition~\ref{prop:descent} and assume that at least one graphical weight of $\omega^*$ is zero.  Then
\[
 N(\omega^*,Q)\ge \frac{3}{\sqrt2}.
\]
\end{corollary}

\begin{proof}
Set $L:=Q^{-1/2}$, so that $\det L=1$ and the metric associated with $L$ through \eqref{eq:lambdaQ} is exactly $Q$.  First observe that the support graph of $\omega^*$ is connected.  Otherwise a nontrivial component cut would contain no positive edge, hence its cut weight would be zero, contradicting $c_i(\omega^*)\ge w_i>0$.  By Proposition~\ref{prop:CNrepresentation}, $LZ_{\omega^*}$ is therefore a three-dimensional parallelohedron.

Moreover, the vanishing of any graphical generator removes at least one facet pair.  Indeed, from \eqref{eq:qi},
\[
 a=0\ \text{or}\ f=0\Rightarrow q_4=af=0,\qquad
 b=0\ \text{or}\ e=0\Rightarrow q_5=be=0,
\]
\[
 c=0\ \text{or}\ d=0\Rightarrow q_6=cd=0.
\]
Since $q_i>0$ is exactly the condition that the corresponding cut direction gives a facet pair, $LZ_{\omega^*}$ has at most six opposite facet pairs, hence at most $12$ facets.  It is therefore not of truncated-octahedral Fedorov type.

By Remark~\ref{rem:boundarysupport} and Proposition~\ref{prop:boundaryRDC},
\[
 N(\omega^*,Q)\ge
 \frac{\Area(LZ_{\omega^*})}{8r(LZ_{\omega^*})^2}
 \ge \frac{12\sqrt2}{8}=\frac{3}{\sqrt2}.
\]
\end{proof}
Thus the only remaining case is an interior six-generator weighting with exactly six tangent facet pairs.

\section{Contact configurations and four-point metrics}
\label{sec:contacts}
 
Define the tangency slacks
\begin{equation}\label{eq:deltai}
  \delta_i(\omega):=c_i(\omega)-w_i\ge0.
\end{equation}
Also define the Hornich--Hlawka deficit
\begin{equation}\label{eq:hdef}
  h:=w_0+w_1+w_2+w_3-w_4-w_5-w_6>0.
\end{equation}
Using the cut identity \eqref{eq:cutrelation} and the definition of $h$ gives
\begin{equation}\label{eq:deltarelation}
  \delta_4+\delta_5+\delta_6
  =\delta_0+\delta_1+\delta_2+\delta_3+h.
\end{equation}

\begin{lemma}\label{lem:slackdouble}
Suppose that exactly six of the seven constraints $c_i(\omega)\ge w_i$ are equalities.  Then the unique nonzero slack $\delta_i$ belongs to one of the three $2+2$ cuts.  After relabelling the vertices of $K_4$, we may assume
\begin{equation}\label{eq:sixcontacts}
  \delta_i=0\quad(i\ne4),\qquad \delta_4=h>0.
\end{equation}
\end{lemma}

\begin{proof}
If the unique positive slack were one of $\delta_0,\dots,\delta_3$, then the left-hand side of \eqref{eq:deltarelation} would vanish while the right-hand side would be strictly positive, a contradiction.  Hence the slack is one of $\delta_4,\delta_5,\delta_6$.  By Remark~\ref{rem:S4} the $S_4$ action is transitive on the three $2+2$ cuts and preserves both positive definiteness and the normalization $\det Q=1$, so we may assume that the slack is $\delta_4$.  Equation \eqref{eq:deltarelation} then gives $\delta_4=h$.
\end{proof}
 
Assume from now on that
\begin{equation}\label{eq:activeexcept4}
  c_i(\omega)=w_i\qquad(i\ne4)\qquad c_4(\omega)>w_4.
\end{equation}
 
We identify the six active lengths $w_i$, $i\neq 4$,  with  the six edges of the $Q$ metric introduced in \eqref{eq:wiN}  restricted to  the four points  $\{0,z_0, z_2, z_6\}$, that is 
\begin{equation}\label{eq:dedges}
\begin{aligned}
\ell_{01}&=\|0-z_0\|_Q=w_0,  &\ell_{23}&=\|z_2-z_6\|_Q=\|-z_1\|_Q=w_1,\\
\ell_{02}&=\|0-z_2\|_Q=w_2, &\ell_{13}&=\|z_0-z_6\|_Q=\|z_3\|_Q=w_3,\\
\ell_{12}&=\|z_0-z_2\|_Q=\|z_5\|_Q=w_5, &\ell_{03}&=\|0-z_6\|_Q=w_6.
\end{aligned}
\end{equation}

We associate to the four-point metric $\ell=(\ell_{ij})$, 
the Gram-type matrices introduced in~\cite[Section 4]{CesaroniNovaga2026}:
\begin{equation}\label{eq:Gamma}
\Gamma(\ell):=
\begin{pmatrix}
 \ell_{01} & \frac{\ell_{01}+\ell_{02}-\ell_{12}}2 & \frac{\ell_{01}+\ell_{03}-\ell_{13}}2\\[0.4em]
 \frac{\ell_{01}+\ell_{02}-\ell_{12}}2 & \ell_{02} & \frac{\ell_{02}+\ell_{03}-\ell_{23}}2\\[0.4em]
 \frac{\ell_{01}+\ell_{03}-\ell_{13}}2 & \frac{\ell_{02}+\ell_{03}-\ell_{23}}2 & \ell_{03}
\end{pmatrix},
\end{equation}
and
\begin{equation}\label{eq:GammaSquared}
 \Gamma(\ell^{\circ2}):=
\begin{pmatrix}
 \ell_{01}^2 & \frac{\ell_{01}^2+\ell_{02}^2-\ell_{12}^2}2 & \frac{\ell_{01}^2+\ell_{03}^2-\ell_{13}^2}2\\[0.4em]
 \frac{\ell_{01}^2+\ell_{02}^2-\ell_{12}^2}2 & \ell_{02}^2 & \frac{\ell_{02}^2+\ell_{03}^2-\ell_{23}^2}2\\[0.4em]
 \frac{\ell_{01}^2+\ell_{03}^2-\ell_{13}^2}2 & \frac{\ell_{02}^2+\ell_{03}^2-\ell_{23}^2}2 & \ell_{03}^2
\end{pmatrix}.
\end{equation}
Recalling that  $w_i^2=z_i^TQz_i$, we get   
\begin{equation}\label{eq:Gammagram}
  \Gamma(\ell^{\circ2})=M_0^TQM_0
\end{equation} where $M_0$ is the $3\times3$ matrix with columns $z_0,z_2,z_6$. 
Explicitly,
\[
 M_0=(z_0\ z_2\ z_6)
 =\begin{pmatrix}1&0&1\\1&1&1\\1&0&0\end{pmatrix},
 \qquad \det M_0=-1.
\]
Hence, since $\det Q=1$,
\begin{equation}\label{eq:detGamma1}
  \det\Gamma(\ell^{\circ2})=1.
\end{equation}

 Because of \eqref{eq:activeexcept4} and \eqref{eq:ci}, the distances in \eqref{eq:dedges} have the explicit form
\begin{equation}\label{eq:identification}
\begin{aligned}
 \ell_{01}&=w_0=a+b+c, &  \ell_{02}&=w_2=b+d+f,  & \ell_{03}&=w_6=a+b+e+f,\\
 \ell_{12}&=w_5=a+c+d+f, & \ell_{13}&=w_3=c+e+f, & \ell_{23}&=w_1=a+d+e.
\end{aligned}
\end{equation} therefore \eqref{eq:Gamma} reads \[ 
\Gamma(\ell)=
\begin{pmatrix}
a+b+c & b & a+b\\[0.4em]
b & b+d+f & b+f \\[0.4em]
 a+b & b+f & a+b+e+f
\end{pmatrix} \qquad  \text{and}\qquad \det\Gamma(\ell)= D(\omega)  
\] where the last equality comes from  \eqref{eq:D} and direct computation.

Therefore, in our geometric configuration the determinant remainder is
\begin{equation}\label{deltaR}
 \mathcal R:=2\det(\Gamma(\ell))^2-\det(\Gamma(\ell^{\circ2}))
 =2D(\omega)^2-1,
\end{equation}
where the last equality uses \eqref{eq:detGamma1}.  It is useful, however, to keep the determinant remainder in its homogeneous form before this normalization.

For the algebraic estimate, use the change of variables
\begin{equation}\label{eq:xmapping}
  x_0=b,\quad x_1=c,\quad x_2=d,\quad x_3=e,
  \quad u=f,\quad v=a,
\end{equation}
and set
\begin{equation}\label{eq:sABC}
\begin{gathered}
 s:=x_0+x_1+x_2+x_3,\qquad
 \Pi:=x_0x_1x_2x_3,\\
 A:=(x_0+x_1)(x_2+x_3),\qquad
 B:=(x_0+x_2)(x_1+x_3), \qquad
 C:=x_0x_3+x_1x_2,
\end{gathered}
\end{equation}
\begin{equation}\label{eq:Mi}
  m_i:=\prod_{j\ne i}x_j\qquad(0\le i\le3),\qquad E_3:=\sum_{i=0}^3 m_i,
\end{equation}
\begin{equation}\label{eq:UVW}
 U:=A^2-8\Pi,\qquad V:=B^2-8\Pi,\qquad W:=C^2-8\Pi,
\end{equation}
and
\begin{align}
\Lambda_u&:=4\bigl[x_0x_1(x_0+x_1)(x_2-x_3)^2
+x_2x_3(x_2+x_3)(x_0-x_1)^2\bigr],\label{eq:Lambdau}\\
\Lambda_v&:=4\bigl[x_0x_2(x_0+x_2)(x_1-x_3)^2
+x_1x_3(x_1+x_3)(x_0-x_2)^2\bigr],\label{eq:Lambdav}\\
B_{21}&:=4\bigl[x_0x_2(x_0+x_2)+x_0x_3(x_0+x_3)
+x_1x_2(x_1+x_2)+x_1x_3(x_1+x_3)\bigr],\label{eq:B21}\\
B_{12}&:=4\bigl[x_0x_1(x_0+x_1)+x_0x_3(x_0+x_3)
+x_1x_2(x_1+x_2)+x_2x_3(x_2+x_3)\bigr].\label{eq:B12}
\end{align}

\begin{lemma} \label{lem:CNremainder}
Let $a,b,c,d,e,f\ge0$, and define the six quantities $\ell_{ij}$ by the right-hand sides of \eqref{eq:identification}.  Then
\begin{align}
&2D(\omega)^2-\det\Gamma(\ell^{\circ2})\notag\\
&\quad=2\sum_{0\le i<j\le3}(m_i-m_j)^2
 +\Lambda_u u+\Lambda_vv\notag\\
&\qquad +2(Uu^2+2Wuv+Vv^2)
 +B_{21}u^2v+B_{12}uv^2\notag\\
&\qquad +2u^2v^2\bigl(2(u+v)^2+4s(u+v)+3s^2\bigr).
\label{eq:CNremainder}
\end{align}
\end{lemma}

\begin{proof}
Substitute \eqref{eq:identification} into \eqref{eq:GammaSquared}, and set
\[
 F:=2D(\omega)^2-\det\Gamma(\ell^{\circ2}).
\]
Both $F$ and the right-hand side of \eqref{eq:CNremainder} are homogeneous polynomials of total degree six in $x_0,x_1,x_2,x_3,u,v$.  Write $[u^iv^j]F$ for the coefficient of $u^iv^j$ when $F$ is viewed as a polynomial in $u,v$.  Expanding the $3\times3$ determinant and collecting powers of $u,v$ gives
\begin{align*}
 [u^0v^0]F&=2\sum_{i<j}(m_i-m_j)^2, & [u]F&=\Lambda_u, & [v]F&=\Lambda_v,\\
 [u^2]F&=2U, & [uv]F&=4W, & [v^2]F&=2V,\\
 [u^2v]F&=B_{21}, & [uv^2]F&=B_{12}.
\end{align*}
The remaining nonzero coefficients are
\[
 [u^4v^2]F=[u^2v^4]F=4,\qquad [u^3v^3]F=8,\qquad
 [u^3v^2]F=[u^2v^3]F=8s,\qquad [u^2v^2]F=6s^2.
\]
All other coefficients vanish.  These are exactly the coefficients obtained by expanding the right-hand side of \eqref{eq:CNremainder}, which proves the identity.
\end{proof}

\begin{proposition}\label{prop:quantremainder}
Let $a,b,c,d,e,f\ge0$ with $b+c+d+e>0$, define $D(\omega)$ by \eqref{eq:D}, and define the six quantities $\ell_{ij}$ by the right-hand sides of \eqref{eq:identification}.  Then
\begin{equation}\label{eq:quantremaindernew}
 2D(\omega)^2-\det\Gamma(\ell^{\circ2})
 \ge \frac{16}{3(b+c+d+e)}\,D(\omega)\,af(be+cd+af).
\end{equation}
If $a,b,c,d,e,f>0$, the inequality is strict.  In particular, since $D(\omega)\ge0$ for nonnegative weights, the left-hand side of \eqref{eq:quantremaindernew} is nonnegative, and it is strictly positive when $a,b,c,d,e,f>0$.
\end{proposition}
\begin{proof}
Set
\[
 \mathcal R_{\mathrm{hom}}:=2D(\omega)^2-\det\Gamma(\ell^{\circ2}).
\]
By Lemma~\ref{lem:CNremainder}, $\mathcal R_{\mathrm{hom}}$ is given by the exact expansion \eqref{eq:CNremainder}.  With the notation above, \eqref{eq:quantremaindernew} reads
\begin{equation}\label{eq:quantremainder}
  \mathcal R_{\mathrm{hom}}\ge \frac{16}{3s}\,D(\omega)\,uv(C+uv).
\end{equation}

We estimate separately the terms of total degree two, three, and at least four in $u,v$ on the right-hand side of \eqref{eq:CNremainder}.  The first line
\[
 2\sum_{i<j}(m_i-m_j)^2+\Lambda_u u+\Lambda_vv
\]
is nonnegative and will not be needed.

\subsubsection*{Step 1: the quadratic sector}
Introduce the two other perfect matchings
\begin{equation}\label{eq:Cprime}
  C':=x_0x_2+x_1x_3,
  \qquad
  C'':=x_0x_1+x_2x_3.
\end{equation}
Then
\begin{equation}\label{eq:AasC}
  A=C+C',\qquad B=C+C''.
\end{equation}
By AM--GM,
\begin{equation}\label{eq:matchingAMGM}
  C,C',C''\ge2\sqrt\Pi.
\end{equation}
Hence
\begin{align}
 U&=(C+C')^2-8\Pi
 =C^2+C'^2+2(CC'-4\Pi)
 \ge C^2+C'^2,\label{eq:Ulower}\\
 V&=(C+C'')^2-8\Pi
 \ge C^2+C''^2.\label{eq:Vlower}
\end{align}
Therefore
\begin{align}
 \sqrt{UV}
 &\ge\sqrt{(C^2+C'^2)(C^2+C''^2)}\notag\\
 &\ge C^2+C'C'',
\label{eq:sqrtUV}
\end{align}
because
\[
 (C^2+C'^2)(C^2+C''^2)-(C^2+C'C'')^2
 =C^2(C'-C'')^2\ge0.
\]
Consequently,
\begin{equation}\label{eq:sqrtUVWpre}
  \sqrt{UV}+W\ge2C^2+C'C''-8\Pi.
\end{equation}

We claim that
\begin{equation}\label{eq:keyquadscalar}
  2C^2+C'C''-8\Pi
  \ge \frac{4E_3C}{3s}.
\end{equation}
Two elementary identities are useful:
\begin{equation}\label{eq:sE3identity}
  sE_3=CC'+CC''+C'C''+4\Pi,
\end{equation}
and
\begin{equation}\label{eq:s2matching}
  s^2-2(C+C'+C'')=x_0^2+x_1^2+x_2^2+x_3^2\ge0.
\end{equation}
To prove \eqref{eq:keyquadscalar}, first suppose $\Pi>0$ and set
\[
 X:=\frac{C}{2\sqrt\Pi},\qquad
 Y:=\frac{C'}{2\sqrt\Pi},\qquad
 Z:=\frac{C''}{2\sqrt\Pi}.
\]
By \eqref{eq:matchingAMGM}, $X,Y,Z\ge1$.  From \eqref{eq:sE3identity} and \eqref{eq:s2matching},
\begin{align*}
 \frac{4E_3C}{3s}
 &=\frac{4C}{3s^2}
   (CC'+CC''+C'C''+4\Pi)\\
 &\le \frac{2C}{3(C+C'+C'')}
   (CC'+CC''+C'C''+4\Pi).
\end{align*}
Thus \eqref{eq:keyquadscalar} follows from
\begin{equation}\label{eq:XYZineq}
  3(X+Y+Z)(2X^2+YZ-2)
  \ge2X(XY+XZ+YZ+1).
\end{equation}
Write $X=1+p$, $Y=1+q$, $Z=1+r$, where $p,q,r\ge0$.  The difference between the two sides of \eqref{eq:XYZineq} expands exactly as
\begin{align}
&1+27p+8q+8r+26p^2+3q^2+3r^2
 +9pq+9pr+13qr\notag\\
&\quad+6p^3+4p^2q+4p^2r+pqr+3q^2r+3qr^2,
\label{eq:positivepoly}
\end{align}
which is strictly positive.  This proves \eqref{eq:keyquadscalar} for $\Pi>0$; the nonnegative boundary follows by continuity.

Combining \eqref{eq:sqrtUVWpre} and \eqref{eq:keyquadscalar},
\begin{equation}\label{eq:sqrtUVW}
  \sqrt{UV}+W\ge\frac{4E_3C}{3s}.
\end{equation}
Since $U,V\ge0$,
\[
 Uu^2+Vv^2\ge2\sqrt{UV}\,uv.
\]
Therefore
\begin{align}
2(Uu^2+2Wuv+Vv^2)
&\ge4(\sqrt{UV}+W)uv\notag\\
&\ge\frac{16E_3C}{3s}uv.
\label{eq:quadraticallocation}
\end{align}
This is exactly the degree-two part required by the right-hand side of \eqref{eq:quantremainder}.

\subsubsection*{Step 2: the cubic sector}
Let
\[
  S_1:=x_0+x_1,\qquad S_2:=x_2+x_3,
\]
so that $A=S_1S_2$ and $S_1+S_2=s$.  From \eqref{eq:B21},
\begin{align*}
 \frac{B_{21}}4
 &=S_2(x_0^2+x_1^2)+S_1(x_2^2+x_3^2)\\
 &\ge \frac12S_2S_1^2+\frac12S_1S_2^2
 =\frac12As.
\end{align*}
Thus
\begin{equation}\label{eq:B21As}
  B_{21}\ge2As.
\end{equation}
Also
\begin{equation}\label{eq:Cbound}
  C=x_0x_3+x_1x_2\le\frac{s^2}{4}.
\end{equation}
For example, if $\theta=x_0+x_3$, then $C\le \theta^2/4+(s-\theta)^2/4\le s^2/4$.  Hence
\begin{equation}\label{eq:B21target}
  B_{21}\ge2As\ge\frac{16CA}{3s}.
\end{equation}
By the analogous grouping of \eqref{eq:B12} with the partition $\{0,2\}|\{1,3\}$,
\begin{equation}\label{eq:B12target}
  B_{12}\ge2Bs\ge\frac{16CB}{3s}.
\end{equation}
Therefore
\begin{equation}\label{eq:cubicallocation}
  B_{21}u^2v+B_{12}uv^2
  \ge\frac{16C}{3s}(Au^2v+Buv^2).
\end{equation}

\subsubsection*{Step 3: total degree at least four}
Put
\[
  \tau:=uv,\qquad m:=u+v.
\]
The last line of \eqref{eq:CNremainder} is
\begin{equation}\label{eq:Rhigh}
  \mathcal R_{\ge4}=\tau^2(4m^2+8sm+6s^2).
\end{equation}
Besides \eqref{eq:Cbound}, we have
\begin{equation}\label{eq:ABe3bounds}
  A\le\frac{s^2}{4},\qquad
  B\le\frac{s^2}{4},\qquad
  E_3\le\frac{s^3}{16}.
\end{equation}
The first two are the elementary product bound $XY\le(X+Y)^2/4$.  The third is Maclaurin's inequality:
\[
 \frac{E_3}{\binom43}\le\left(\frac{s}{4}\right)^3,
 \qquad\text{hence }E_3\le\frac{s^3}{16}.
\]
Consequently,
\begin{align}
&\frac{16\tau^2}{3s}(sC+E_3+Au+Bv+s\tau)\notag\\
&\qquad\le \tau^2\left(\frac53s^2+\frac43sm+\frac{16}{3}\tau\right).
\label{eq:highRHSbound}
\end{align}
Since $m^2=(u+v)^2\ge4uv=4\tau$,
\begin{align*}
4m^2+8sm+6s^2
&-\left(\frac53s^2+\frac43sm+\frac{16}{3}\tau\right)\\
&=4m^2-\frac{16}{3}\tau+\frac{20}{3}sm+\frac{13}{3}s^2\\
&\ge\frac{32}{3}\tau+\frac{20}{3}sm+\frac{13}{3}s^2\ge0.
\end{align*}
Together with \eqref{eq:Rhigh} and \eqref{eq:highRHSbound},
\begin{equation}\label{eq:highallocation}
  \mathcal R_{\ge4}
  \ge\frac{16\tau^2}{3s}(sC+E_3+Au+Bv+s\tau).
\end{equation}

\subsubsection*{Step 4: recombination}
From \eqref{eq:D} and \eqref{eq:xmapping},
\[
 D(\omega)=E_3+Au+Bv+suv.
\]
Expanding the desired right-hand side gives
\begin{align}
\frac{16}{3s}D(\omega) uv(C+uv)
&=\frac{16E_3C}{3s}uv
 +\frac{16C}{3s}(Au^2v+Buv^2)\notag\\
&\quad+\frac{16u^2v^2}{3s}
 (sC+E_3+Au+Bv+suv).
\label{eq:RHSSplit}
\end{align}
The three displayed groups are bounded respectively by \eqref{eq:quadraticallocation}, \eqref{eq:cubicallocation}, and \eqref{eq:highallocation}.  All omitted terms of \eqref{eq:CNremainder} are nonnegative.  This proves \eqref{eq:quantremainder}.

If all six variables $x_0,x_1,x_2,x_3,u,v$ are positive, the polynomial in \eqref{eq:positivepoly} is strictly positive and $uv>0$, so the quadratic allocation is strict.  Hence \eqref{eq:quantremainder} is strict.
\end{proof}

\section{The determinant--tangency deficit inequality}
\label{sec:deficit}
In this section we show that the estimate \eqref{eq:Ntarget} holds at $\omega$ in the interior of the positive cone, with  $c_i(\omega)\ge w_i$, if exactly six constraints are active. This is the remaining terminal case needed for the proof.

We first convert the quantitative remainder \eqref{eq:quantremaindernew} into a comparison between the determinant deficit and the unique tangency slack.

\begin{proposition}\label{prop:deficit}
Assume $a,b,c,d,e,f>0$, $\det Q=1$, $c_i(\omega)\ge w_i$, and
\[
  c_i(\omega)=w_i\qquad(i\ne4).
\]
Then
\begin{equation}\label{eq:deficitcorr}
  D(\omega)-\frac1{\sqrt2}
  >\frac13q_4(\omega)(c_4(\omega)-w_4).
\end{equation} 
\end{proposition}

\begin{proof}
By \eqref{eq:detGamma1}, the homogeneous remainder of Proposition~\ref{prop:quantremainder} specializes in the present geometric configuration to
\[
 \mathcal R=2D(\omega)^2-1,
\]
with $\mathcal R$ as in \eqref{deltaR}.  Since $a,b,c,d,e,f>0$, the right-hand side of \eqref{eq:quantremaindernew} is positive, so $\mathcal R>0$; this uses only Proposition~\ref{prop:quantremainder}.  In particular $D(\omega)>1/\sqrt2$ and hence $D(\omega)+1/\sqrt2\le2D(\omega)$.  Therefore
\begin{equation}\label{eq:factorbound}
 D(\omega)-\frac1{\sqrt2}
 =\frac{\mathcal R}{2(D(\omega)+1/\sqrt2)}
 \ge\frac{\mathcal R}{4D(\omega)}.
\end{equation}
Proposition~\ref{prop:quantremainder} now gives
\begin{equation}\label{eq:deficitintermediate}
 D(\omega)-\frac1{\sqrt2}
 \ge \frac{\mathcal R}{4D(\omega)}
 > \frac{4af(be+cd+af)}{3(b+c+d+e)},
\end{equation}
where strictness follows from $a,b,c,d,e,f>0$.

On the other hand, the explicit cut normals \eqref{eq:z} satisfy the matrix identity
\begin{equation}\label{eq:z4matrixidentity}
 z_4z_4^T=\sum_{i=0}^3 z_i z_i^T-z_5z_5^T-z_6z_6^T.
\end{equation}
Contracting \eqref{eq:z4matrixidentity} with $Q$ yields
\[
 w_4^2=w_0^2+w_1^2+w_2^2+w_3^2-w_5^2-w_6^2.
\]
Since $w_i=c_i(\omega)$ for $i\ne4$, we obtain
\[
  w_4^2=c_0(\omega)^2+c_1(\omega)^2+c_2(\omega)^2+c_3(\omega)^2-c_5(\omega)^2-c_6(\omega)^2.
\]
Substituting the expressions in \eqref{eq:ci} gives
\[w_4^2=(b+c+d+e)^2-4(be+cd+af).\]
Since \eqref{eq:ci} gives $c_4(\omega)=b+c+d+e$, we deduce
\[
 c_4(\omega)-w_4=(b+c+d+e)-w_4=\frac{(b+c+d+e)^2-w_4^2}{(b+c+d+e)+w_4}\leq \frac{4(be+cd+af)}{b+c+d+e}.
\]
Since $q_4(\omega)=af$ by \eqref{eq:qi}, we conclude 
\[
 \frac13q_4(\omega)(c_4(\omega)-w_4)
 \le\frac{4af(be+cd+af)}{3(b+c+d+e)}.
\]
Together with \eqref{eq:deficitintermediate} this proves \eqref{eq:deficitcorr}.
\end{proof}

\begin{corollary}\label{cor:sixcontact}
Fix $Q>0$ with $\det Q=1$, and let $a,b,c,d,e,f>0$.  Under the normalization $c_i(\omega)\ge w_i$, if exactly six constraints are active, then
\[
  N(\omega,Q)>\frac3{\sqrt2}.
\]
\end{corollary}

\begin{proof}
By Lemma~\ref{lem:slackdouble}, after relabelling the unique slack is $\delta_4=c_4(\omega)-w_4>0$.  Hence
\[
 \sum_{i=0}^6q_i(\omega)(c_i(\omega)-w_i)=q_4(c_4(\omega)-w_4).
\]
Using \eqref{eq:EulerIdentity} and Proposition~\ref{prop:deficit},
\begin{align*}
 N(\omega)
 &=\sum_iq_i(\omega)w_i
 =\sum_iq_i(\omega)c_i(\omega)-\sum_iq_i(\omega)(c_i(\omega)-w_i)\\
 &=3D(\omega)-q_4(\omega)(c_4(\omega)-w_4)\\
 &> 3D(\omega)-3\left(D(\omega)-\frac1{\sqrt2}\right)
 =\frac3{\sqrt2}.
\end{align*}
\end{proof}

\begin{corollary}\label{cor:interiorstrict}
Fix $Q>0$ with $\det Q=1$.  If $a,b,c,d,e,f>0$ and $c_i(\omega)\ge w_i$ for all $i$, then
\[
 N(\omega,Q)>\frac3{\sqrt2}.
\]
\end{corollary}

\begin{proof}
If exactly six cut constraints are active at $\omega$, this is Corollary~\ref{cor:sixcontact}.  Otherwise at most five constraints are active, and Proposition~\ref{prop:descent} produces a feasible $\omega^*$ with
\[
 N(\omega^*,Q)<N(\omega,Q)
\]
such that either exactly six constraints are active at $\omega^*$, or some graphical weight of $\omega^*$ vanishes.  In the first case Corollary~\ref{cor:sixcontact} gives $N(\omega^*,Q)>3/\sqrt2$, in the second Corollary~\ref{cor:boundaryfull} gives $N(\omega^*,Q)\ge3/\sqrt2$.  In either case
\[
 N(\omega,Q)>N(\omega^*,Q)\ge\frac3{\sqrt2}. \qedhere
\]
\end{proof}

\section{Proof of the main result}
\label{sec:mainproof}

\begin{proof}[Proof of Theorem~\ref{thm:main}]
The non-truncated types are covered by Proposition~\ref{prop:boundaryRDC}.  We may therefore assume that $P$ has truncated-octahedral Fedorov type.  In the graphical representation all six weights are then positive; equivalently, the associated cell has fourteen facets, as in \cite[Proposition~3.3]{CesaroniNovaga2026}.  Concretely, if any graphical weight vanished, one of $q_4=af$, $q_5=be$, or $q_6=cd$ would vanish by \eqref{eq:qi}, so at least one of the seven opposite facet pairs would disappear.

Since the quotient $\Area(P)/r(P)^2$ is invariant under similarities, we may replace $P$ by a similar representative and write $P=Q^{-1/2}Z_\omega$ with $\det Q=1$.  Using the rescaling invariance of $\mu^2N$ noted above, rescale $\omega$ so that $\mu=1$.  By Lemma~\ref{lem:inradius},
\[
 c_i(\omega)\ge w_i\qquad(0\le i\le6),
\]
and the normalized body has inradius $1/2$.  By Lemma~\ref{lem:surface}, its surface area is $2N(\omega,Q)$.  Corollary~\ref{cor:interiorstrict} therefore yields
\[
 N(\omega,Q)>\frac3{\sqrt2},
\]
and hence
\[
 \frac{\Area(P)}{r(P)^2}=8N(\omega,Q)>12\sqrt2
\]
for every truncated-octahedral parallelohedron.  Since $\mu^2N$ is invariant under rescaling $\omega$, this is the original, unnormalized quotient in \eqref{eq:ratioMuN}.

It remains to classify equality.  The preceding paragraph excludes equality throughout the truncated-octahedral stratum.  Proposition~\ref{prop:boundaryRDC} shows that among the four non-truncated Fedorov types equality occurs if and only if $P$ is similar to the regular rhombic dodecahedron.  This proves both \eqref{eq:mainhom} and its equality statement.
\end{proof}

\section*{Acknowledgements}
The authors are members of the Gruppo Nazionale per l'Analisi Matematica, la Probabilità e le loro Applicazioni (GNAMPA) of the Istituto Nazionale di Alta Matematica (INdAM).

\end{document}